\documentclass[11pt,reqno]{amsart}

\parskip=\smallskipamount

\newtheorem{theorem}{Theorem}[section]
\newtheorem{lemma}[theorem]{Lemma}

\theoremstyle{definition}

\newtheorem{remark}[theorem]{Remark}

\newcommand{\bea}{\begin{eqnarray*}}
\newcommand{\eea}{\end{eqnarray*}}

\numberwithin{equation}{section}

\begin{document}

\author{J. E. Forn\ae ss}
\author{Erlend F. Wold}

\title[]{An estimate for the Squeezing function and estimates of invariant metrics}
%
%
\subjclass[2000]{}
\date{\today}

\begin{abstract}
We give estimates for the squeezing function on strictly pseudoconvex domains, and derive 
some sharp estimates for the Carath\'eodory, Sibony and Azukawa metrics near their boundaries.

\end{abstract}

\maketitle


\section{Introduction}

Let $\Omega $ be a bounded domain in $\mathbb C^n.$ The squeezing function \cite{DengGuanZhang} measures how much a domain looks like the unit ball
observed from a given point $z$.   More precisely it is defined as follows:  For a given 
injective holomorphic map $f:\Omega\rightarrow\mathbb B^n$ satisfying 
$f(z)=0$ we set 
$$
S_{\Omega,f}(z):={\mathrm{sup}}\{r>0:r\mathbb B^n\subset f(\Omega)\},
$$
and then we set 
$$
S_\Omega(z):=\underset{f}{\mathrm{sup}}\{S_{\Omega,f}(z)\},
$$
where $f$ ranges over all injective holomorphic maps $f:\Omega\rightarrow\mathbb B^n$
with $f(z)=0$.  It was proved in \cite{DengGuanZhang} that 
$$
\underset{z\rightarrow b\Omega}{\lim}S_\Omega(z)=1
$$
if $\Omega$ is a $\mathcal C^2$-smooth strictly pseudoconvex domain, and 
it was proved in \cite{KimZhang} that the squeezing function is bounded on any bounded 
convex domain.    Our goal 
is to improve this estimate when the boundary has higher regularity, and to 
give an application to invariant metrics.  

\begin{theorem}\label{squeezing}
Let $\Omega=\{\delta<0\} \subset\mathbb C^n$ be a strictly pseudoconvex domain with a defining function $\delta$ of class $\mathcal C^k$ for 
$k\geq 3$.  
The squeezing function $S_\Omega(z)$ for $\Omega$
satisfies the estimate 
$$
S_\Omega(z)\geq 1- C\cdot\sqrt{|\delta(z)|}
$$
for a fixed constant $C$.
If we even have  $k\geq 4$, then
there exists a constant $C>0$ such that the squeezing function $S_\Omega(z)$ for $\Omega$
satisfies 
$$
S_\Omega(z)\geq 1- C\cdot|\delta(z)|
$$
for all $z$ 
\end{theorem}


Combining with a theorem due to D. Ma \cite{Ma} and 
a result of Deng, Guan and Zhang \cite{DengGuanZhang}, 
an immediate consequence is a sharp estimate 
for invariant metrics near the boundary of a
strictly pseudoconvex domain.   Before we state the result, 
we briefly recall the definitions of some invariant metrics. Let $\Delta$ denote the unit disc, and let $\mathcal O(M,N)$ denote the holomorphic maps from $M$ to $N.$ 

\begin{itemize}
\item Kobayashi metric $K_\Omega(p,\xi).$ We define
$$
K_\Omega(p,\xi)=\inf \{|\alpha|; \exists f\in \mathcal O(\Delta,\Omega)\; f(0)=p, \alpha f'(0)= \xi \}.
$$

\item Carath\'eodory metric $C_\Omega(p,\xi).$ We define
$$
C_\Omega(p,\xi)=\sup \{|f'(p)(\xi)|; \exists f\in \mathcal O(\Omega,\Delta)\; f(p)=0\}.
$$

\item Sibony metric $S_\Omega(p,\xi).$ We define

$S_\Omega(p,\xi)=\sup \{(\sum_{i,j} \frac{\partial^2 u(p)}{\partial z_i\partial \overline{z}_j}\xi_i\overline{\xi}_j)^{1/2},
u(p)=0, 0\leq u<1,$
 $u$ is $\mathcal C^2$ near $p$ and $\ln u$ is plurisubharmonic in $\Omega$\}.

\item Azukawa metric $A_U(p,\xi).$ We define
$$
A_\Omega(p,\xi)=\sup_{u\in P_\Omega(p)} \{\limsup_{\lambda \searrow 0} \frac{1}{|\lambda| }u(p+\lambda \xi )\}
$$ 
where
\bea
P_\Omega(p) & =  & \{u:\Omega\rightarrow [0,1), \ln u \; {\mbox{is plurisubharmonic and}} \\
&  \exists & M_u>0,
  r_u>0 \;
{\mbox{such that}} \\
& &  \mathbb B^n(p,r)\subset \Omega, u(z)\leq M\|z-p\|, z\in \mathbb B^n(p,r)\}
\eea
\end{itemize}

\begin{theorem}\label{metrics}  
Let $\Omega\subset\mathbb C^n$ be a strictly pseudoconvex domain of class $\mathcal C^3$, 
let $p\in b\Omega$, and let $\delta$ be a defining function for $\Omega$ near $p$, 
such that $\|\nabla\delta(z)\|=1$ for all $z\in b\Omega$.  Then if $F_\Omega(z,\zeta)$
is either the Carath\'eodory, Sibony or Azukawa metric, there exists a constant $C>0$
such that 
\begin{align*}
& (1-C\sqrt{|\delta(z)|})\left[\frac{L_{\pi(z)}(\xi_T)}{|\delta(z)|} + \frac{\|\xi_N\|}{4\delta(z)^2}\right]^{1/2}\leq F_\Omega(z,\xi)\\
& \leq (1+C\sqrt{|\delta(z)|})\left[\frac{L_{\pi(z)}(\xi_T)}{|\delta(z)|} + \frac{\|\xi_N\|}{4\delta(z)^2}\right]^{1/2}
\end{align*}
for all $z$ near $p$, and all $\xi=\xi_N+\xi_T$, where $\pi$ is the orthogonal projection
to $b\Omega$, $\xi_N$ is the complex normal component of $\xi$ at $\pi(z)$ and $\xi_T$ is the complex tangential component, and $L$ is the Levi form of $\delta$. 
\end{theorem}

Ma's result is the corresponding statement for the Kobayashi metric, and the result is sharp in the sense that one cannot in general do 
better than the square root of the boundary distance.


\medskip

%
%

\section{Proof of Theorem \ref{metrics}}\label{metricsection}

The following was proved in \cite{DengGuanZhang}, and we include the proof for 
the benefit of the reader. 
\begin{lemma}\label{compare}
Let $\Omega$ be any bounded domain in $\mathbb C^n$, and let $F_\Omega(z,\xi)$ be either the Carath\'eodory, Sibony or Azukawa metric.  
Then 
$$
S_\Omega(z)K_\Omega(z,\xi)\leq F_\Omega(z,\xi)\leq K_\Omega(z,\xi)
$$
for all $z\in\Omega$ and all $\xi\in\mathbb C^n$, where $K$ denotes the Kobayashi metric.
\end{lemma}
\begin{proof}
It is well known that $K$ dominates $F$ so we need to show the lower estimate.  
Let $f:\Omega\rightarrow\mathbb B^n$ be injective holomorphic with $f(z)=0$, such that 
$B_{r}\subset f(\Omega)$ where $r=S_\Omega(z)$.  For the existence of $f$ see \cite{DengGuanZhang} (alternatively one can use a limiting argument).
We get that 
\begin{align*}
F_\Omega(z,\xi) & = F_{f(\Omega)}(0,f_*\xi)\geq F_{\mathbb B^n}(0,f_*\xi)=K_{\mathbb B^n}(0,f_*\xi)\\
& = S_{\Omega}(z) K_{B_r}(0,f_*\xi)\geq S_\Omega(z)K_{f(\Omega)}(0,f_*\xi)=S_\Omega(z)K_\Omega(z,\xi).
\end{align*}
\end{proof}

\emph{Proof of Theorem \ref{metrics}:}
By Lemma \ref{compare} we have that 
$$
S_\Omega(z)K_\Omega(z,\xi)\leq F_\Omega(z,\xi)\leq K_\Omega(z,\xi)
$$
Then combining Theorem \ref{squeezing} with the fact that Theorem \ref{metrics} holds with $F_\Omega(z)$ replaced
by $K_\Omega(z)$ (see \cite{Ma}) completes the proof.
$\hfill{\square}$


\section{Proof of Theorem \ref{squeezing}}

The following provides the key geometric setup for the proof. Let $k=3$ or $4$, and
let $\Omega$ be a bounded strongly pseudoconvex domain of class $\mathcal C^k.$ 

\begin{lemma}\label{exposedpoint}
Let $p\in b\Omega$.  There exists an injective holomorphic map $\phi:\overline\Omega\rightarrow\mathbb C^n$
such that $\tilde\Omega=\phi(\Omega)$ satisfies the following: 
\begin{itemize}
\item[(i)] $\tilde\Omega\subset\mathbb B^n$, 
\item[(ii)] $\phi(p)=(1,0,\cdot\cdot\cdot,0)=:a$ and $\phi^{-1}(b\mathbb B^n)=\{p\}$, 
\item[(iii)] near a we have that, $\tilde\Omega=\{\rho<\mu^2\}, 0<\mu<1$ where  
$$
\rho(z)=|z_1-(1-\mu)|^2+\|z'\|^2 +  O(|z_1-1|^2) + O(\|z-a\|^k).
$$
\end{itemize}
\end{lemma}

\begin{proof}
By the main theorem in \cite{DiederichFornaessWold} there exists a map $\phi$
such that (i) and (ii) are satisfied.  That we can achieve (iii) 
follows from the proof which 
consists of three steps.  We first apply an automorphism of $\mathbb C^n$ 
to ensure that, locally near $p=0$, our domain has a defining function 
\begin{equation}\label{localform1}
\rho(z)=2Re(z_1) + \|z\|^2 + O(\|z\|^k).
\end{equation}
To achieve this one approximates a local map with jet interpolation using 
the Anders\'{e}n-Lempert theory.  
We next apply another automorphism of $\mathbb C^n$ which can be chosen 
to match the identity at the origin to any given order, so we still have a defining 
function of the form \eqref{localform1}.  The final exposing map 
is of the form $\varphi=\phi\circ\alpha$, where $\phi(z)=(f(z_1),z_2,...,z_n)$
where $f$ is injective holomorphic with $f'(0)>0$, 
and $\alpha(z)$ can be chosen to match the identity 
to any given order 
at the origin.  By a translation we assume that $\varphi(0)=0$.  We then have 
a defining function for $\varphi(\Omega)$ of the form 
\begin{align*}
\rho(z) & =2Re(c_1z_1 + c_2 z_1^2 + c_3z_1^3) + |c_1|^2|z_1|^2 + \|z'\|^2 + O(|z_1|^2) +O(\|z\|^k) \\
& = 2c_1Re(z_1) + |c_1|^2|z_1|^2 + \|z'\|^2 + O(|z_1|^2) + O(\|z\|^k).
\end{align*}
Applying the linear change of coordinates $(z_1,z')\mapsto (z_1/c_1,z')$, we get a defining function
$$
\rho(z)=2Re(z_1) + |z_1|^2 + \|z'\|^2 + O(|z_1|^2) + O(\|z\|^k).
$$
By chosing a small $0<\mu<1$ we have that $\mu\varphi(\Omega)$ is contained 
in the translated unit ball $\{2Re(z_1)+\|z\|^2<0\}$, with defining function 
$$
\rho(z)=2\mu Re(z_1) + |z_1|^2 + \|z'\|^2 + O(|z_1|^2) + O(\|z\|^k), 
$$
which is the same as (iii) when translated $(z_1,z')\mapsto (z_1+1,z')$.
\end{proof}

\begin{remark}
On $b\tilde\Omega$ the remainder term in (iii) is actually $O(|z_1-1|^{k/2})$.
To see this we first translate $\tilde\Omega$ to the origin, set $\tilde{z}_1=z_1-1,\tilde{z}=(\tilde{z}_1,z')$ so that 
it is defined by 
$$
\tilde\rho(\tilde{z})=2Re(\tilde{z}_1)+|\tilde{z}_1|^2+\frac{1}{\mu}\|z'\|^2 + O(|\tilde z_1|^2) +O(\|\tilde{z}\|^k)<0.
$$
We estimate $\|z'\|$ on $\tilde\rho=0$.  If $\|z'\|\leq |\tilde{z}_1|$
the remainder term is less than $C|\tilde{z}_1|^k= O(|z_1|^{k/2})$.  
If $|\tilde{z}_1|\leq\|z'\|$ then the remainder term is $O(\|z'\|^k)$ and we get 
\begin{align*}
\|z'\|^2+ O(\|z'\|^k) & =\mu(-2Re(\tilde{z}_1) - |\tilde{z}_1|^2 + O(|\tilde z_1|^2)) \\
& =\mu |\tilde{z}_1|(-\frac{2Re(\tilde{z}_1)}{|\tilde{z}_1|} - |\tilde{z}_1| + \frac{O(|\tilde z_1|^2)}{|\tilde z_1|}).
\end{align*}
This implies that the remainder term is $\mathcal O(\|z'\|^k)=\mathcal O(|z_1-1|^{k/2}).$
\end{remark}

From now on we assume that $\Omega=\tilde{\Omega}$ and satisfies (i)-(iii) above.  Then 
$\Omega$ is "almost" contained in the ball $B_\mu\subset\mathbb B^n$
defined by 
$$
|z_1|^2 + \frac{1}{\mu}\|z'\|^2<1.
$$ 
We will use automorphisms of the ball $\mathbb B^n$ of the form 
$$
\phi_r(z_1,z')=\left(\frac{z_1-r}{1-rz_1},\frac{\sqrt{1-r^2}}{1-rz_1}z'\right).
$$
We have that $\phi_r$ leaves $B_\mu$ invariant.  To prove the 
theorem, we will estimate two things: 

\begin{itemize}
\item[(a)] How much  $\phi_r(\Omega)$ sticks out of $B_\mu$ and
\item[(b)]  the size of the largest ball 
in $B_\mu$ contained in $\phi_r(\Omega).$ 
\end{itemize}

\subsection{Estimate (a)}

\begin{lemma}\label{outsideestimate}  
There exists a constant $C>0$ such that for $w\in b\phi_r(\Omega)$ we have 
that $|w_1|^2+\frac{1}{\mu}\|w'\|^2\leq 1 + C(1-r)^{\frac{k-2}{2}}$.
\end{lemma}

\begin{proof} 
We would like to express the maximum of the function $\|\phi_r(z)\|$
in terms of $(1-r)$ on $b\Omega$, \emph{i.e.}, we look at 
$$
\|\phi_r(z)\|^2=\frac{|z_1-r|^2 + \frac{1}{\mu}(1-r^2)\|z'\|^2}{|1-rz_1|^2}=\frac{|z_1-r|^2}{|1-rz_1|^2} + \frac{1}{\mu}\frac{(1-r^2)|z'|^2}{|1-rz_1|^2}
$$
for $z\in b\Omega$.  Fix any $\eta>0.$ We show first that if $z\in \mathbb B^n$ with $|z_1-1|>\eta,$
then we have a uniform estimate
$$
\|\phi_r(z)\|^2\leq 1+C(1-r).
$$

In this case we have that the denominator of the second term stays bounded independent of $r$, while 
$|z'|\leq 1$, hence the term goes to zero like $(1-r)$.  For the other term we write 
$$
\frac{|z_1-r|^2}{|1-rz_1|^2}  = 1 + \frac{(1-r^2)(|z_1|^2-1)}{|1-rz|^2}\leq 1+C(1-r).
$$

Next we look at $|z_1-1|\leq\eta$.  If $\eta$ is chosen small enough, the local description (iii)
is valid.   Hence if  $|z_1-1|<\eta$ and if $z\in b\Omega$ we 
have that 
\begin{align*}
\|z'\|^2 & = - |z_1-(1-\mu)|^2 + O(|z_1-a|^{k/2}) + \mu^2 \\
& = -|z_1-1|^2 - 2\mu Re(z_1-1) + O(|z_1-a|^{k/2}), \\ 
\end{align*}  
which gives that 
\begin{align*}
\frac{1}{\mu}\|z'\|^2 & = -\frac{1}{\mu}|z_1-1|^2 - 2Re(z_1-1) +  O(|z_1-a|^{k/2}) \\
& \leq -|z_1-1|^2 - 2Re(z_1-1) + O(|z_1-a|^{k/2}) \\
& = 1 - |z_1|^2 + O(|z_1-a|^{k/2}).
\end{align*}

Hence 
\begin{align*}
\frac{|z_1-r|^2 + \frac{1}{\mu}(1-r^2)\|z'\|^2}{|1-rz_1|^2} & = \frac{|z_1-r|^2 + (1-r^2)(1-|z_1|^2)}{|1-rz_1|^2} \\
& + \frac{(1-r^2)O(|z_1-1|^{k/2})}{ |1-rz_1|^2}\\
& = 1 + \frac{(1-r^2)O(|z_1-1|^{k/2})}{ |1-rz_1|^2}\\
& \leq 1+C\frac{(1-r^2) |1-rz_1|^{k/2}}{ |1-rz_1|^2}\\
& \leq 1+C_1 \frac{1-r}{|1-rz_1|^{2-(k/2)}}\\
& \leq 1+C_1 \frac{1-r}{(1-r)^{2-(k/2)}}\\
& \leq 1 + C_2(1-r)^{\frac{k-2}{2}}.
\end{align*}

\end{proof}

\subsection{Estimate (b)}

We define 
$$
B^{\mu}_{\eta,\tilde{\eta}}=\{|z_1-(1-\eta)|^2+\frac{\tilde\eta}{\mu}|z'|^2<\eta^2\}
$$
with constants $0<\eta\leq \tilde{\eta}<2\eta.$

\begin{lemma}\label{ballinside} We set 
$\tilde{\eta} = \begin{cases}
		\eta, & k=4 \\
		\frac{\eta}{1-C\eta},& k=3
	\end{cases}$\\
(i) If $k=4$ then $B^{\mu}_{\eta,\tilde{\eta}}\subset \Omega$ for all $\eta$ small enough\\
(ii) If $k=3$, 
and the constant $C>0$ is fixed large enough, then $B^{\mu}_{\eta,\tilde{\eta}}\subset \Omega$ for all $\eta$ small enough.
\end{lemma}

\begin{proof}
For $\eta$ small enough, the ellipsoid $B^{\mu}_{\eta,\tilde{\eta}}$ is contained in the region where the local defining function $\rho$ is defined. Since $\rho$ is plurisubharmonic it suffices to show that $\rho\leq 0$ on
$bB^\mu_{\eta,\tilde{\eta}}.$ We translate coordinates, by setting $\tilde{z}_1=z_1-1$ and
$\tilde{z}=(\tilde{z}_1,z').$ We want to show that 

$$
\{2\eta Re(\tilde{z}_1) + |\tilde{z}_1|^2+\frac{\tilde\eta}{\mu} |z'|^2=0\}
$$
 is contained in the set 
$$
\{2\mu Re(\tilde{z}_1)+|\tilde{z}_1|^2+\|z'\|^2+ O(|\tilde z_1|^2)+O(\|\tilde{z}\|^k)\leq 0\}.
$$
Write $\tilde{z}_1=x_1+iy_1$.
On the boundary of the ellipsoid we have that 
$$
2\eta x_1 + x_1^2+y_1^2 + \frac{\tilde\eta}{\mu}\|z'\|^2=0\Leftrightarrow\frac{\mu}{\tilde\eta}y_1^2+\|z'\|^2=-\frac{\mu}{\tilde\eta}(2\eta x_1 + x_1^2)
$$
and consequently we get on the boundary of the ellipsoid that 
\begin{align*}
\|\tilde{z}\|^2 & = x_1^2+y_1^2 + \|z'\|^2 \leq x_1^2 + \frac{\mu}{\tilde\eta}y_1^2 + \|z'\|^2\\
& = x_1^2 - \frac{\mu}{\tilde\eta}(2\eta x_1 + x_1^2)\\
& = -x_1(-x_1 + \frac{\mu}{\tilde\eta}(2\eta + x_1))
\end{align*}

It follows that $\|\tilde{z}\|^2\leq C|x_1|$, and so 
\begin{equation}\label{sizez}
\|\tilde{z}\|^k\leq C |x_1|^{k/2}
\end{equation}

Consider again the boundary of the ellipsoid; we have 
$$
x_1^2+y_1^2+2\eta x_1+\frac{\tilde{\eta}}{\mu} \|z'\|^2=0
$$
Hence
$$
\|z'\|^2=-\frac{\mu}{\tilde{\eta}}(x_1^2+y_1^2+2\eta x_1)
$$

Therefore
\begin{align*}
2\mu x_1+|\tilde{z}_1|^2+\|z'\|^2 + O(|\tilde z_1|^2) +O(\|\tilde{z}\|^k)
& \leq 2\mu x_1 +|\tilde{z}_1|^2-\frac{\mu}{\tilde{\eta}}(|\tilde{z}_1|^2\\
& +2\eta x_1)+C_3 |x_1|^{k/2} + C_4|\tilde z_1|^2.
\end{align*}
using \eqref{sizez}. 
It suffices therefore to show that the right side is $\leq 0.$  This means: 
\begin{equation}\label{goal}
2\mu x_1(1-\frac{\eta}{\tilde\eta}) + |\tilde{z}_1|^2(1-\frac{\mu}{\tilde\eta}) + C_3|x_1|^{k/2} + C_4|\tilde z_1|^2 \leq 0.
\end{equation}

Observe that $\frac{1}{2}\leq \frac{\eta}{\tilde{\eta}}\leq 1$, so $1-\frac{\eta}{\tilde{\eta}}\geq 0.$
Morever $x_1\leq 0$ on the translated ellipse. Hence the first term in \eqref{goal} is $\leq 0.$ It suffices therefore
that 

\begin{equation}
|\tilde{z}_1|^2(1-\frac{\mu}{\tilde{\eta}}) + C|\tilde z_1|^{k/2}  \leq 0,
\end{equation}
where we merged the constants $C_3$ and $C_4$.
When $k=4$, this holds as soon as $\eta$ is small enough. When $k=3$, this holds when

$$C|\tilde{z}_1|^{3/2}  \leq \frac{\mu}{\tilde{\eta}}|\tilde{z}_1|^{1/2}|\tilde{z}_1|^{3/2} (1-\frac{\tilde{\eta}}{\mu})
$$

or 
$$C  \leq \frac{|\tilde{z}_1|^{1/2}}{\tilde{\eta}} (\mu-\tilde{\eta})
$$

This holds when $|\tilde{z}_1|\geq \tilde{C} \eta^2$ for large enough $\tilde{C}.$
To complete the proof we need to consider the case when $k=3$ and $|\tilde{z}_1|\leq \tilde{C} \eta^2$, 
and we go back to consider the full expression \eqref{goal}.   Since the sum 
$|\tilde z_1|^2(1-\frac{\mu}{\tilde\eta})+C_4|\tilde z_1|^2$ is negative when 
$\eta$ is small, it is enough to determine when 
$$
2\mu x_1(1-\frac{\eta}{\tilde{\eta}})+C_3|x_1|^{3/2} \leq 0.
$$
or equivalently when
$$
2\mu x_1(1-\frac{\eta}{\tilde{\eta}}) \leq  C_3x_1|x_1|^{1/2}\Leftrightarrow 2\mu (1-\frac{\eta}{\tilde{\eta}}) \geq C_3|x_1|^{1/2}.
$$
By our assumption we now have that $C_3|x_1|^{1/2}\leq C_3(\tilde C\eta^2)^{1/2}=C_5\eta$, and 
so we need that 
$$
2\mu(1-\frac{\eta}{\tilde\eta})\geq C_5\eta.
$$
Hence the choice $\tilde{\eta}=\frac{\eta}{1-\frac{C_5}{2\mu} \eta}$ works.

\end{proof}

Now let $\psi(z_1,z')=(z_1,\frac{1}{\sqrt\mu}z')$.  Then $\psi(B^\mu_{\eta,\tilde\eta})$
is the ellipsoid  
$$
B^1_{\eta,\tilde\eta}=\{|z_1-(1-\eta)|^2+\tilde\eta\|z'\|^2<\eta^2\}, 
$$

\begin{lemma}\label{ballinside3} Let $0<\eta,r<1$ and $\tilde{\eta}>0.$ 
If $z\in bB^1_{\eta,\tilde {\eta}},$ then
\bea
\|\phi_r(z_1,z')\|^2  & =1 + \frac{(1-r^2)|z_1-1|^2}{|1-rz|^2}  - \frac{(1-r^2)(1/{\tilde\eta})|z_1-1|^2}{|1-rz_1|^2}\\
&+ \frac{(1-r^2)2(1-\frac{\eta}{\tilde\eta})(Re(z_1)-1)}{|1-rz_1|^2} \\
\eea
\end{lemma}

\begin{proof}

\begin{align*}
\|\phi_r(z_1,z')\|^2 & = \frac{|z_1-r|^2 + (1-r^2)|z'|^2}{|1-rz_1|^2}\\
& = \frac{|z_1-r|^2 + (1-r^2)(1/{\tilde\eta})(\eta^2- |z_1-(1-\eta)|^2)}{|1-rz_1|^2}\\
& = \frac{|z_1-r|^2 + (1-r^2)(1/{\tilde\eta})(\eta^2- |z_1-1|^2 - 2\eta Re(z_1-1)-\eta^2)}{|1-rz_1|^2}\\
& = \frac{|z_1-r|^2 + (1-r^2)(1/{\tilde\eta})(-2\eta Re(z_1-1)-|z_1-1|^2)}{|1-rz_1|^2}\\
& = 1 + \frac{|z_1-r|^2 - |1-rz_1|^2- (1-r^2)\frac{2\eta}{\tilde\eta}Re(z_1-1)}{|1-rz_1|^2} \\
& - \frac{(1-r^2)(1/{\tilde\eta})|z_1-1|^2}{|1-rz_1|^2}\\
& = 1 + \frac{|z_1|^2 -2rRe(z_1) + r^2 - (1 - 2rRe(z_1) + r^2|z_1|^2)}{|1-rz_1|}\\ 
& -\frac{ (1-r^2)\frac{2\eta}{\tilde\eta}Re(z-1)}{|1-rz_1|^2}
- \frac{(1-r^2)(1/{\tilde\eta})|z_1-1|^2}{|1-rz_1|^2}\\
& = 1 + \frac{(1-r^2)(|z_1|^2-\frac{2\eta}{\eta}Re(z_1)+(\frac{2\eta}{\tilde\eta}-1))}{|1-rz_1|^2}\\
& - \frac{(1-r^2)(1/{\tilde\eta})|z_1-1|^2}{|1-rz_1|^2}\\
& = 1 + \frac{(1-r^2)|z_1-1|^2}{|1-rz|^2}  - \frac{(1-r^2)(1/{\tilde\eta})|z_1-1|^2}{|1-rz_1|^2}\\
& + \frac{(1-r^2)2(1-\frac{\eta}{\tilde\eta})(Re(z_1)-1)}{|1-rz_1|^2} \\
\end{align*}

\end{proof}

\begin{lemma}\label{ballinside2}
Let $\psi(z)=(z_1,\frac{1}{\sqrt\mu}z')$.   Suppose that $0<\eta,r<1, 1-2\eta<r$ and $\tilde{\eta}>0.$
Then 
$\psi(\phi_r(B^\mu_{\eta,\tilde\eta}))$ contains the ball of radius
$$
\sqrt{1 - 2(1-r)\frac{1}{\tilde\eta} - 4|1-\frac{\eta}{\tilde\eta}|}.
$$

\end{lemma}

\begin{proof}
Since $1-2\eta<r,$ we have that $0\in \psi(\phi_r(B^\mu_{\eta,\tilde\eta})).$  Hence it suffices to show that
$ \|\psi(\phi_r)(z)\|^2 \geq 1 - 2(1-r)\frac{1}{\tilde\eta} - 4|1-\frac{\eta}{\tilde\eta}|$ on the boundary of $B^\mu_{\eta,\tilde\eta}.$
(This is nonempty if the expression on the right is nonnegative.)
Since $\psi\circ \phi_r=\phi_r\circ \psi$, it suffices to show that
$ \|\phi_r(z)\|^2 \geq 1 - 2(1-r)\frac{1}{\tilde\eta} - 4|1-\frac{\eta}{\tilde\eta}|$ on the boundary of $B^1_{\eta,\tilde\eta}.$
From the previous lemma we have that
\bea
\|\phi_r(z)\|^2 & \geq & 1 + \frac{(1-r^2)|z_1-1|^2}{|1-rz|^2}  - \frac{(1-r^2)(1/{\tilde\eta})|z_1-1|^2}{|1-rz_1|^2}\\
& + & \frac{(1-r^2)2(1-\frac{\eta}{\tilde\eta})(Re(z_1)-1)}{|1-rz_1|^2} \\
 & \geq & 1  - \frac{(1-r^2)(1/{\tilde\eta})|z_1-1|^2}{|1-rz_1|^2}
 - \frac{(1-r^2)2|1-\frac{\eta}{\tilde\eta}||Re(z_1)-1|}{|1-rz_1|^2}\\
& \geq & 1-\frac{(2(1-r))(1/{\tilde\eta})|rz_1-1|^2}{|1-rz_1|^2}
 - \frac{(2(1-r))2|1-\frac{\eta}{\tilde\eta}||rz_1-1|}{|1-rz_1|^2}\\
& \geq & 1-(1-r))(2/{\tilde\eta})
 - \frac{4|1-rz_1|)(1-\frac{\eta}{\tilde\eta})}{|1-rz_1|}\\
& \geq & 1-(1-r))(2/{\tilde\eta})
 - 4|1-\frac{\eta}{\tilde\eta}|\\
\eea

\end{proof}

We prove Theorem \ref{squeezing}

\begin{proof}
We will estimate the squeezing function at points $(r,0)$ when $r<1$ is close to $1.$
That this gives the uniform constant claimed in  Theorem 1.1, follows from the
dependence on $p$ as $p$ varies over the boundary of the original domain. In particular, the constants
in our estimates can be chosen independently of the point $p$, and 
the radial lines
will foliate a neighborhood of the boundary so that we get an estimate for all points near the boundary.
The map $\psi\circ \phi_r$ maps $(r,0)$ to the origin. We estimate the image of $\Omega.$

It follows from Lemma \ref{outsideestimate}  that there exists a constant $C>0$ such that for $w\in b\phi_r(\Omega)$ we have 
that $|w_1|^2+\frac{1}{\mu}\|w'\|^2\leq 1 + C(1-r)^{\frac{k-2}{2}}$.
Since the left side is plurisubharmonic, the same estimate holds by the maximum principle on 
$\overline{\phi_r(\Omega)}.$
Suppose that $(z_1,z')\in \psi(\overline{\phi_r(\Omega)}).$ Then $(z_1,z')=\psi(w_1,w')=(w_1,\frac{1}{\sqrt{\mu}}w')$
for some $w\in \overline{\phi_r(\Omega)}.$
Hence $\|z\|^2=|w_1|^2+\frac{1}{\mu}\|w'\|^2\leq 1 + C(1-r)^{\frac{k-2}{2}}.$
It follows that $\psi(\phi_r(\Omega))$ is contained in the ball centered at the origin of radius
$1 + C(1-r)^{\frac{k-2}{2}}$.

We next estimate the radius of the largest ball contained in $\psi(\phi_r(\Omega)).$
By Lemma \ref{ballinside} we have ellipsoids  $B^{\mu}_{\eta,\tilde{\eta}}=\{|z_1-(1-\eta)|^2+\frac{\tilde\eta}{\mu}|z'|^2<\eta^2\}$ contained in $\Omega$ for certain $\eta, \tilde{\eta}:$ We set

$\tilde{\eta} = \begin{cases}
		\eta, & k=4 \\
		\frac{\eta}{1-C\eta},& k=3.
	\end{cases}$\\
(i) If $k=4$ we have that $B^{\mu}_{\eta,\tilde{\eta}}\subset \Omega$ for all $\eta$ small enough, and \\
(ii) if $k=3$, 
and the constant $C>0$ is fixed large enough, then $B^{\mu}_{\eta,\tilde{\eta}}\subset \Omega$ for all $\eta$ small enough.
We can then estimate instead the largest ball contained in $\psi(\phi_r(B^{\mu}_{\eta,\tilde{\eta}})).$

We use Lemma \ref{ballinside2}:
   Suppose that $0<\eta,r<1, 1-2\eta<r$ and $\tilde{\eta}>0.$
Then 
$\psi(\phi_r(B^\mu_{\eta,\tilde\eta}))$ contains the ball of radius
$$
\sqrt{1 - 2(1-r)\frac{1}{\tilde\eta} - 4|1-\frac{\eta}{\tilde\eta}|}.
$$

We deal first with the case $k=4.$ Then we assume that $1-2\eta<r$ and $\tilde{\eta}=\eta.$ It follows that

$\psi(\phi_r(\Omega)) \supset \psi(\phi_r\left(B^\mu_{\eta,\tilde\eta}\right)
\supset \mathbb B(0,\sqrt{1-2(1-r)\frac{1}{\tilde{\eta}}}).$
We choose a fixed $\eta$, and let $r\rightarrow 1.$ We then get that for a fixed constant $C'$,
$\psi(\phi_r(\Omega))\supset \mathbb B(0,1-C'(1-r)).$
Hence we have shown that in the case $k=4, \frac{k-2}{2}=1,$
$$
\mathbb B(0,1-C'(1-r))\subset \psi(\phi_r(\Omega)) \subset \mathbb B(0,1+C(1-r)).
$$
Composing with the map $\lambda(z)=\frac{z}{1+C(1-r)}$ we obtain that
$\lambda(\psi(\phi_r(r,0)))=0$ and that
$$
\mathbb B(0,\frac{1-C'(1-r)}{1+C(1-r)}) \subset \lambda(\psi(\phi_r(\Omega))) \subset \mathbb B(0,1).
$$

Hence it follows that the squeezing function at $(r,0)$ is at least $1-C''(1-r).$
Since the defining function  $\delta(z)=-(1-r)+\mathcal O(1-r)^2$ for $z=(r,0)$ and $r$ close to $1$, we obtain Theorem 1.1 in the case when $k=4.$

It remains to do the case $k=3.$

It follows as above that $\psi(\phi_r(\Omega))$ is contained in the ball centered at the origin of radius
$1 + C(1-r)^{\frac{k-2}{2}}=1 + C(1-r)^{\frac{1}{2}}$.

As above we 
suppose that $0<\eta,r<1, 1-2\eta<r$, and we have that
$\psi(\phi_r(B^\mu_{\eta,\tilde\eta}))$ contains the ball of radius
$$
\sqrt{1 - 2(1-r)\frac{1}{\tilde\eta} - 4|1-\frac{\eta}{\tilde\eta}|}.
$$
We have that $\frac{\tilde\eta}{\eta}=1-C\eta$, and so it follows that

\bea
\psi(\phi_r(\Omega)) & \supset & \psi(\phi_r(B^\mu_{\eta,\tilde\eta})\\
& \supset & \mathbb B\left(0,\sqrt{1-2(1-r)\frac{1}{\tilde{\eta}}-4C\eta}\right)\\
& \supset  & \mathbb B\left(0,\sqrt{1-2(1-r)\frac{1}{\eta}-4C\eta}\right).\\
\eea

In this case, we let $\eta$ depend on $r.$ Set $\eta=\sqrt{1-r}$.
Then $r=1-\eta^2>1-2\eta$  if $r$ is close enough to $1.$
We then get that 

\bea
\psi(\phi_r(\Omega)) 
& \supset  & \mathbb B\left(0,\sqrt{1-2(1-r)\frac{1}{\eta}-4C\eta}\right)\\
& = & \mathbb B\left(0,\sqrt{1-2(1-r)\frac{1}{\sqrt{1-r}}-4C\sqrt{1-r}}\right)\\
& = & \mathbb B\left(0,\sqrt{1-(2+4C)\sqrt{1-r}}\right)\\
& \supset & \mathbb B\left(0,1-(2+4C)\sqrt{1-r}\right)\\
\eea

Now it follows by the same scaling type argument with a map $\lambda$ that
we get the desired lower bound for the squeezing function in the case $k=3.$

\end{proof}

\section{An example}

Let $\Omega$ be the domain $\Omega:=\mathbb B^n\setminus\frac{1}{2}\overline{\mathbb B}^n$.
We will show that $S_\Omega(z)$ cannot approach 1 faster than $1-C\mathrm{dist}(z,b\Omega)$.
By abuse of notation we set $r=(r,0,...,0), 0<r<1$ and we set $a=(1/2,0,...,0)$.  Then the Kobayashi
distance with respect to $\mathbb B^n$ from $a$ to $r$ is $\frac{1}{2}(\log(\frac{1+r}{1-r}) - \log(3))$.
Now let $f:\Omega\rightarrow\mathbb B^n$ be an injective holomorphic map with $f(r)=0$.
Then $f$ extends to a holomorphic map $\tilde f:\mathbb B^n\rightarrow\mathbb B^n$, so 
by the decreasing property of the Kobayashi metric we have that the Kobayashi 
distance between $f(r)$ and $f(a)$ is less that $\frac{1}{2}\log(\frac{1+r}{1-r})$.  It 
follows that $S_{\Omega,f}(r)\leq r$.

\bibliographystyle{amsplain}

\end{document}